\documentclass{article}

\usepackage[paperwidth=169mm, paperheight=239mm]{geometry}
\usepackage{setspace}
\usepackage{amsbsy,amsmath,amsfonts,amssymb}
\usepackage{amsthm}
\usepackage{graphicx,color}
\usepackage{enumerate}

\newcommand{\stsets}[1]{\mathbb{#1}}
\newcommand{\R}{\stsets{R}}

\newcommand{\N}{\stsets{N}}
\newcommand{\Z}{\stsets{Z}}

\renewcommand{\P}{\mathbf{P}}
\newcommand{\E}{{\bf E}\,}
\newcommand{\var}{{\bf var}\,}

\DeclareMathOperator{\card}{card}
\DeclareMathOperator{\one}{{1\hspace*{-0.55ex}I}}
\newcommand{\cond}{\hspace*{1ex} \rule[-1ex]{0.15ex}{3ex}\hspace*{1ex}}

\newcommand{\thru}{,\dotsc,}

\renewcommand{\vec}[1]{\mathbf{#1}}
\newcommand{\ti}{\to\infty}
\newcommand{\ssp}{\hspace{0.5pt}}
\newcommand{\seg}{see, \hbox{e.\ssp g.,}\ }
\newcommand{\ie}{\hbox{i.\ssp e.}\ }

\newcommand{\iid}{\hbox{i.\ssp i.\ssp d.}\ }

\renewcommand{\epsilon}{\varepsilon}
\renewcommand{\phi}{\varphi}
\newcommand{\eps}{\varepsilon}

\newcommand{\ppp}{\mathcal{P}}
\newcommand{\fff}{\mathcal{F}}

\newlength{\querylen}
\usepackage{fancybox}

\newenvironment{AMS}{\vspace*{0.5cm}\noindent {\bf
    AMS 2010 Subject Classification:} } {\vspace{0.5cm}}

\newenvironment{Ack}{\vspace*{1.5cm} \begin{center}{\bf
    Acknowledgements}\end{center}\vspace*{0cm}\nopagebreak
   \rm } {\vspace{0.5cm}}

\newcommand{\keywords}[1]
  {\begin{center}
  \begin{minipage}{315.83pt}
  \small
  \noindent \emph{Keywords:}~{\textrm{#1}}
  \end{minipage}
  \end{center}
  \normalsize
  }

\newtheorem{thm}{\noindent Theorem}[section]
\newtheorem{lem}{\noindent Lemma}[section]

\newtheorem{rem}{\noindent Remark}[section]

\newcommand{\BF}{\mathrm{BF}}
\newcommand{\DB}{\mathrm{DB}}

\newcommand{\td}{\stackrel{\mathcal{D}}{\to}}
\newcommand{\sF}{\mathcal{F}}

\newtheorem{Th A1}{Theorem A1}
\newtheorem{Th A2}{Theorem A2}
\newtheorem{Th B}{Theorem B}

\newcommand{\n}{\eta} 
\newcommand{\ttot}{\nu_{\text{total}}}

\newcommand{\czeta}{\check{\zeta}}
\newcommand{\cchi}{\check{\chi}}

\newcommand{\taugre}{\tau_{>k}}

\newcommand{\Bern}{\ensuremath{\mathsf{Bern}}}

\newcommand{\tzeta}{\widetilde{\zeta}}

\newcommand{\stleq}{\leq_{\mathrm{st}}}

\begin{document}

\title{Bit flipping and time to recover}

\author{Anton Muratov\thanks{Chalmers University of Technology,
    Department of Mathematical Sciences, SE-412 96 Gothenburg, Sweden. Email: \texttt{[muratov|sergei.zuyev]@chalmers.se}} \and 
\addtocounter{footnote}{-1} Sergei Zuyev\footnotemark}

\maketitle

\begin{abstract}
  We call `bits' a sequence of devices indexed by positive integers,
  where every device can be in two states: $0$ (idle) and $1$
  (active). Start from the `ground state' of the system when all bits are
  in $0$-state. In our first Binary Flipping (BF) model, the evolution of the system is the
  following: at each time step choose one bit from a given
  distribution $\ppp$ on the integers independently of anything else,
  then flip the state of this bit to the opposite.  In our second
  Damaged Bits (DB) model a `damaged' state is added: each selected
  idling bit changes to active, but selecting an active bit changes
  its state to damaged in which it then stays forever.

  In both models we analyse the recurrence of the system's ground state when
  no bits are active. We present sufficient conditions for both BF and DB
  models to show recurrent or transient behaviour, depending on the properties
  of $\ppp$. We provide a bound for fractional moments of the return time to
  the ground state for the BF model, and prove a Central Limit Theorem for the
  number of active bits for both models.

\end{abstract}

\smallskip
\keywords{binary system; bit flipping; random walk on a countable
  group; Markov chain recurrence; critical behaviour}

\smallskip
\begin{AMS} Primary 60J27; Secondary 60J10; 68Q87 \end{AMS} 

\newpage

\section{Introduction and Model Description}
\label{sec:introduction}

In many areas of engineering and science one faces an array of devices
which possess a few states. In the simplest case these could be on-off
or idle-active states, in other situations a damaged state is also
possible. By the analogy with computer science, such a two-state
device can be called a \emph{bit} which in some case can also be
`damaged'. If the activation-deactivation cycles (flipping) or damage
produce themselves in a random fashion, a natural question to ask is
when, if at all, the system of bits recovers to some initial or
\emph{ground state} when none of the bits are active, allowing only
for idling and damaged bits to be seen. The time to recover may be
finite, but, in general, may also assume infinite values when the
system actually does not recover. In the latter case we speak of
\emph{transient} behaviour of the system. In the former case,
depending on whether the mean of the recover time exists or not, we
speak of a \emph{positive-} or of a \emph{null-recurrence}. Similarly
to random walk models, this classification is tightly related to the
exact random mechanism governing the change of the bits' states.

In the present paper we consider two basic models. In both models we
deal with a countably infinite array of bits which we index by the
positive integers $\N=\{1,2,\dots\}$. Initially, at step 0, the system
is in the ground state, i.e.\ all the bits are idling. At each next
step the index of the bit to change its state is sampled independently
of the current state of the bits from a given
probability distribution on $\N$,
\begin{displaymath}
\ppp=(p_1,p_2,\dots):  \sum_{i=1}^\infty p_i = 1.
\end{displaymath}
Without loss of generality, we may
assume that the bits are indexed in such a way that  
\begin{displaymath}
p_1\geq p_2\geq p_3 \geq \dots
\end{displaymath}
so that the bits most likely to change their state are
put first. We also assume that the support of the distribution $\ppp$
is unbounded, otherwise our models are described by a finite state
Markov chain with an evident behaviour. The main quantities of
interest are the number of steps $\tau$ until the first return to the
ground state, and $\n_n$ -- the number of bits being active at step
$n$. The two models are the following.

\paragraph{Binary Flipping (BF).}
\label{sec:binary-shoot-bs}
In this model the bits alternate between the two states: idle and
active. At step 0 all of the bits are idling. Let $\chi_1, \chi_2,\dots$
be \iid variables taken from the distribution $\ppp$. At the $i$th
step, $i=1,2,\dotsc$, the bit with index $\chi_i$ is \emph{flipped},
i.e.\ its state is changed to the opposite:
\begin{quote}
\emph{idle}$\ \leftrightarrow\ $\emph{active}
\end{quote}
If 0 and 1 represent, respectively, the idling and the active states,
the evolution of the system is described by a discrete time Markov
chain $\{\zeta_n\}_{n\geq 0}=
\{(\zeta^1_n,\zeta^2_n,\zeta^3_n,\dotsc)\}_{n \geq 0}$ with the state
space
\begin{displaymath}
\mathcal{X} = \left\{x \in \{0,1\}^\N:\ x \text{ has finitely many non-zeros}\right\}
\end{displaymath}
such that
$\zeta_0=\vec{0}$ is the zero-vector, and
\begin{equation}\label{eq:bfmc}
  \zeta^k_{n+1}=\begin{cases}
    \zeta^k_n, & k\neq \chi_{n+1},\\
    1-\zeta^k_n, & k=\chi_{n+1},
  \end{cases}\quad k=1,2,\dotsc ,n=0,1,2,\dotsc.
\end{equation}

The main quantity of interest is the number of steps required for
the system to return to the ground state, i.e.\ the following stopping time:
\begin{align*}
  \tau_{\mathrm{BF}} & = \min\{n\geq 1: \text{no bits are active at
    step\ } n\} \\
  & =\min\{n\geq 1: \zeta^k_n=0,\quad \forall k=1,2,\dotsc\}\,.
\end{align*}

\paragraph{Damaged Bits (DB).}
\label{sec:shoot-kill-sk}
This second model elaborates on the first one by adding a damaged
state to the bits. As in BF model above, we start with a sequence of idling bits and then
consecutively sample from $\ppp$ for the index of the bit to change
its state according to the following dynamics:
\begin{quote}
\emph{idle}$\ \to\ $\emph{active}$\ \to\ $\emph{damaged}.
\end{quote}
Thus in this model, the reversal of states is not possible: once a bit is
active it will never become idle again. An attempt to activate an already
active bit leads to its damage. Also, the damaged bits never
become functional again and if a damaged bit is selected to change
its state nothing happens: it just remains damaged.

If 0,1,2 encode idle, active and damaged states
respectively, the corresponding Markov chain $\{\zeta_n\}_{n\geq 0}$
with the state space
\begin{displaymath}
  \mathcal{Y} = \left\{y \in \{0,1,2\}^\N: y \text{ has finitely many non-zeros}\right\}
\end{displaymath}
is defined by
\begin{equation}\label{eq:dbmc}
  \zeta^k_{n+1}=\begin{cases}
    \zeta^k_n, & k\neq \chi_{n+1},\\
    \min\{2,1+\zeta^k_n\}, & k=\chi_{n+1},
  \end{cases}\quad k,n\in\N
\end{equation}
with the starting configuration $\zeta_0$ being the vector of all zeroes, $\zeta_0 = \vec{0}$.

Here again we are looking for the number of steps to return to the
ground state which is now understood as the collection of all of the
states without active bits:
\begin{align*}
  \tau_{\mathrm{DB}} & =\min\{n\geq 1: \text{no bits are active
  at step\ } n\}\\
& = \min\{n\geq 1: \zeta^k_n\in\{0,2\},\quad \forall k=1,2,\dotsc\}\,.
\end{align*}
In contrast to BF model, the ground state in DB model in general cannot be
identified with any one particular state of the Markov chain $\{\zeta_n\}$.

\paragraph{Continuous Time Version.}
So far we have formulated the discrete time dynamics of the system of
bits. It is also sensible to consider continuous-time versions of
both BF and DB models. Let $\zeta_t=(\zeta^1_t,\zeta^2_t,\dotsc)$ be a
sequence of continuous-time Markov jump processes, each with the state space
$\{0,1\}$ in the BF case, and with $\{0,1,2\}$ in the DB case. The
$k$th process $\{\zeta^k_t\}_{t\geq 0}$ represents the corresponding change of
states of the $k$th bit which happens with exponentially distributed
holding times at rate $p_k$. Note that since all $p_k$
sum up to 1, there is an a.s.\ finite number of state changes of the
whole system of bits in any finite period of time. Therefore one can
define the renewal process $\{t_n\}$ of times when some of the bits
changes its state. The embedded Markov chain $\{\zeta_{t_n}\}_{n\geq 0}$
is then a distributional copy of the discrete-time
version $\{\zeta_n\}_{n\geq 0}$ of the model. One of the advantages
of this representation, also known as \emph{Poissonisation} and widely used
since at least \cite{AthKar:68}, is the independence of $\zeta_t^k$
for different $k=1,2,\dotsc$. This often leads to explicitly
computable probabilities as we also demonstrate here.  Further we use
the discrete time and the continuous time versions of the models
interchangeably, whichever is more convenient at the moment: the
notion of recurrence/transience stays the same for both.
  
The Markov chains \eqref{eq:bfmc} and \eqref{eq:dbmc} describing our
models can be regarded as random walks on an infinite-dimensional
group, \seg~\cite{LyoPemPer:96}. Typically the analysis of random
walks on discrete groups assumes a finite generator set, so that the
underlying Cayley graph is locally finite, as for example,
in~\cite{Woe:00}. However, the state spaces in our models are not
finitely generated groups, so analysis of a random walk in such a
space is interesting in its own right. But practical applications are
also envisaged: in addition to an evident relation to modelling
reliability of a complex system with multiple components prone to fail
at different rates, one can also mention computer science and
information encryption techniques. The very term ``Bit Flipping'' is
borrowed from the literature on randomised simplex algorithms
\cite{BalPem:07}, where a similar model was analysed: each flipped bit
there makes all of the bits to the right change their states as
well. This model applies to estimate the running speed of a random
edge simplex algorithm on a Klee-Minty cube which is particularly
`bad' for many optimisation algorithms thus providing a worst-case
scenario, \seg \cite{DezNemTer:08} and the references therein.

Finally, we mention and interesting interpretation of the BF model as
a dynamical percolation process on $\Z$, where we start with all edges
'open', and then they start 'closing' independently of each other,
each with different rate. The question of recurrence is then
equivalent to the question of existence of a sequence of percolation
times when all the edges are open and thus 0 is connected to the
infinity. For a recent survey on the dynamical percolation, see
\cite{steif2009survey}.

\section{Main Results}

For the above models we prove the following main result: each model
exhibits a transient or recurrent behaviour, depending on how fast
$p_k$'s decay. There is a critical decay separating both regimes,
different for each model. We start characterising the critical
decay in the Binary Flipping model.

\begin{thm}\label{th:bf-gen-suf}
  If the distribution $\ppp$ is such that:
  \begin{enumerate}
  \item[(i)] $\limsup\limits_{k\ti}2^kp_k <\infty$, then BF model is
  recurrent, \ie $\P(\tau_{\mathrm{BF}}<\infty)=1$,\\
  \item[(ii)] $\liminf\limits_{k\ti}(2-\eps)^kp_k >0$ for
  some $\eps>0$, then BF model is
  transient, \ie $\P(\tau_{\mathrm{BF}}=\infty)>0$.
  \end{enumerate}
\end{thm}

Loosely speaking, the critical decay of $p_k$'s in BF model is the geometric
distribution with parameter $\frac{1}{2}$. Although deterministic systems may
behave rather differently from stochastic ones, often in a non-critical regime
they provide a good intuition to what is happening. Imagine an infinite row of
lamps, turning on and off with deterministic frequencies $p_k=(1/2)^{k-1}$.
That means, the first lamp changes its state every second, the second lamp
every 2 seconds, the third every 4 seconds, etc., meaning that this row is
nothing else than a digital clock showing the time since the start in a binary
format. Then at least one lamp is lit at every positive time instant. This is
still true when $p_k=p^k$ with $p>1/2$: the $(n+1)$th lamp will always turn on
before the $n$th turns off, so the active intervals of $n$th and $(n+1)$th
lamps will overlap for every $n$. Thus this deterministic system never returns
to the ground state whenever $p\geq 1/2$. However, for $p<1/2$ the first $n$
lamps will have time to run through all possible combinations (including all
zeroes) before the $(n+1)$th lamp will be turned on, so there will always be
an infinite number of occurrences of the ground state when no lamp is lit. As
Theorem~\ref{th:bf-gen-suf} shows, the same critical decay separates the
stochastic BF model too.

Furthermore, the BF model
is never positive recurrent, as the next theorem shows.
\begin{thm}
  \label{th:null-rec}
  When a BF model is recurrent, it is null-recurrent, \ie
  $\E\tau_{\mathrm{BF}}=\infty$ always.
\end{thm}

This result can be easily foreseen by regarding the BF process as an
irreducible time-reversible Markov chain (\ref{eq:bfmc}). The
time-reversibility implies that the stationary measure is uniform, but
the state space is countably infinite, hence it cannot be
probabilistic so the chain cannot be positive recurrent.

Although the first moment of $\tau_{\mathrm{BF}}$ is infinite, it is reasonable to
ask for which values of $r<1$ the $r$th moment becomes finite. The next
theorem presents bounds for such $r$ in the case of asymptotically geometrically
decaying $\{p_k\}$, these are presented graphically on Figure~\ref{fig:bounds}.

\begin{thm}\label{th:bf-moments}
  Consider the recurrent BF model in discrete time with
  $p_k\sim C_1 p^k$ for some fixed constant $C_1>0$ and $p\in (0, 1/2)$.  Then
  \begin{itemize}
  \item[(i)] $\E\tau_\BF^r<\infty$ for any positive $r<1-\frac{\log 2}{\log(1/p)}$.
    Moreover, for any such $r$, if the Markov chain~(\ref{eq:bfmc}) is
    started from an arbitrary $\zeta_0$ with the largest active bit $M_0$,
    then there exists a constant $C_2=C_2(C_1,p,r)$ such that
    \begin{displaymath}
      \E\left[\tau_\BF^r | M_0=m\right] \leq C_2\left(\frac{1}{2p}\right)^m;
    \end{displaymath}
  \item[(ii)] $\E\tau_\BF^r=\infty$ for any $r>1-\frac{\log(2-p)}{\log(1/p)}$.
  \end{itemize}
\end{thm}

\begin{figure}[ht]
  \centering
  \input{Fig1_t}
  \caption{Integrability of $\tau^r_{\mathrm{BF}}$ as
    given by Theorem~\ref{th:bf-moments}. 
    \label{fig:bounds}}
\end{figure}

\begin{rem}
  There is an obvious coupling of the DB model with the BF model: just declare
  the bits which flipped more than once in BF model damaged in DB. Then
  $\tau_\DB\leq\tau_\BF$ almost surely and the same upper bound \emph{(i)} of
  Theorem~\ref{th:bf-moments} is also true for $\tau_\DB$.
\end{rem}

The DB model can also be recurrent or transient, depending on $p_k$.
The recurrence/transience of the model now does not correspond to recurrence/transience
of the Markov chain~\eqref{eq:dbmc}, because the ground state
of the DB model is an infinite collection of states of $\{\zeta_n\}$. Still, we call
the DB model recurrent, if $\tau_{\DB}<\infty$ with probability 1, and
transient otherwise. Denote by $Q_k$ the tail of the distribution $\ppp$:
\begin{displaymath}
Q_k = \sum_{j=k+1}^\infty p_k.
\end{displaymath}

\begin{thm}\label{th:db-gen-exp}
  If the distribution $\ppp$ is such that:
  \begin{enumerate}[(i)]
  \item $\limsup\limits_{k\to\infty} \frac{Q_{k+1}}{Q_k} = p < 1$,
  then the DB model is recurrent,
  \item $p_k \sim C\exp(-\alpha k^\gamma), k\to\infty$ for
  some $\alpha>0, \gamma\in(0,1/2)$, then the DB model is transient.
  \end{enumerate}
\end{thm}

Denote by $\n_t$ the total number of active bits in the continuous version
of the model at time $t\geq 0$. In both
BF and DB models, whenever $\E\n_t\to \infty,$
conditions of the Central
Limit Theorem are fulfilled for $\n_t$. We prove the following fact:
\begin{thm}\label{th:bs-db-clt}
  For both BF and DB models, whenever 
  \begin{equation}\label{eq:infexp}
    \E\n_t\to \infty,
  \end{equation} then also
  $\var \n_t\to\infty$ as $t\to\infty$ and
  \begin{displaymath}
    \frac{\n_t-\E\n_t}{\sqrt{\var\n_t}} \td \mathcal{N}(0,1)\quad\text{as}\ t\to\infty.
  \end{displaymath}
  In BF model the condition~\eqref{eq:infexp} is always fulfilled,
  and in DB model a sufficient condition for~\eqref{eq:infexp} is:
  \begin{equation}\label{eq:sufdbinfexp}
    p_k \sim C \exp(-\alpha k^\gamma), k\to \infty,
  \end{equation}
  for some constants $C>0, \alpha>0, \gamma \in (0,1)$.
\end{thm}

\begin{rem}
  In the above theorem, both $\E\n_t$ and $\var\n_t$ admit an explicit
  form of a series:
  \begin{align*}
    \E\n_t & = \sum_{k=1}^\infty f(p_kt), &
    \var\n_t = \sum_{k=1}^\infty f(p_kt)(1-f(p_kt)),
  \end{align*}
  where $f(x)=(1-e^{-x})/2$ for BF and $f(x)=xe^{-x}$ for DB model. In
  both cases,
  $f(p_kt)$ is the probability for the $k$th bit to be active at
  time $t$ in the corresponding model.
\end{rem}

\section{Proofs}

\subsection{Transience and recurrence of BF model}
\label{sec:solution-bs-model}

\begin{proof}[Proof of Theorem~\ref{th:bf-gen-suf}]

  First, we are going to prove the theorem for a particular case of
  $p_k=Cp^{k}$ for some $p\in(0,1)$ and then extend it using monotonicity
  arguments. 
  
  Consider the continuous-time BF model. Recall $\zeta_t=(\zeta^k_t)_{k\geq 1}$,
  a continuous-time
  Markov jump process on $\mathcal{X}$ representing the configuration of
  the bits at time $t\geq 0$, and $\zeta_0=\vec{0}=(0,0,\dots)$,
  see~(\ref{eq:bfmc}).  Denote by $\ttot$ the total time $\{\zeta_t\}$
  spends in the state $\vec{0}$ for $t>0$. Since the process
  $\{\zeta_t\}$ is irreducible, recurrence of the BF model implies that
  the state $\vec{0}$ is recurrent. Since the holding times at state
  $\vec{0}$ are i.i.d.\ exponential with parameter 1, we get
  $\E\ttot=\infty$. When the BF model is transient, i.e.\ when
  \begin{displaymath}
    q=\P\{\zeta_t=\vec{0}\ \text{for some finite } t>t_1\cond
    \zeta_0=\vec{0}\}<1, 
  \end{displaymath}
  where $t_1$ is the time of the first jump of the process $\zeta_t$,
  then $\ttot$ is distributed as the sum $\sum_{i=1}^\nu \epsilon_i$,
  where $\nu$ has geometrical distribution with parameter $q$ and
  $\epsilon_i$'s are i.i.d.\ exponentially distributed with parameter 1 
  r.v.'s representing holding times at state $\vec{0}$. In that
  case, $\E\ttot=\E\nu\,\E\epsilon_i=1/q<\infty$. Thus
  $\E\ttot=\infty$ is equivalent to recurrence of $\zeta(t)$ and of
  the BF model.

One can write
\begin{displaymath}
  \E\ttot  =\E\int\limits_0^\infty\prod_{k=1}^\infty\one\{k\text{th
      bit is idle at time\ }t\}\, dt
  =\int\limits_0^\infty\prod_{k=1}^\infty\P\{\zeta^k_t=0\}\,dt.
\end{displaymath}
Next, 
\begin{align*}
  \P\{\zeta^k_t=0\}&=\sum_{j=0}^\infty\P\{k\text{th bit flipped\ }2j\text{
  times by time\ }t\}\\ &=e^{-p_kt}\sum_{j=0}^\infty \frac{(p_kt)^{2j}}{(2j)!}
  =(1+e^{-2p_kt})/2,
\end{align*}
thus the transience is equivalent to the convergence of the integral
\begin{equation}
  \label{eq:1}
  \E\ttot=\int_0^\infty \prod_{k=1}^\infty (1+e^{-2p_k t})/2\, dt\,.
\end{equation}
In the second part of the proof we provide the lower and upper bounds for the infinite product
under the integral.

Denote by $f(x)=(1-e^{-2x})/2$, so that the product under the integral in
~\eqref{eq:1} becomes $\prod\limits_{k=1}^\infty (1-f(p_kt))$.

Fix an arbitrary small $\eps>0$. Note that the function $1-f(x)$ is monotone decreasing in $x$ and the equation
\begin{displaymath}
  1-f(x)=\frac{1}{2-\eps}
\end{displaymath}
has the only root. Call this root $z_\eps$. Now, represent the product
as a multiplication of the two factors:
\begin{displaymath}
  \prod\limits_{k=1}^\infty (1-f(p_kt)) = \underbrace{\prod\limits_{k:p_kt<z_\eps}
  (1-f(p_kt))}_{\Phi_1(t)} \underbrace{\prod\limits_{k:p_kt\geq z_\eps} (1-f(p_kt))}_{\Phi_2(t)},
\end{displaymath}

First, note that every term of the product $\Phi_1(t)$ is less or equal
than $1$, therefore, $\Phi_1(t)\leq 1$.  Next, observe that
$\{p_k\}_{k\geq 1}$ is the geometric distribution, and therefore
$\Phi_1(t)=\Phi_1(\frac{t}{p^n})$ for each $n=1,2,\dotsc$, moreover,
taking into account that the function $f$ is continuous,
non-increasing, we obtain
\begin{align*}
  \Phi_1(t) &= \prod_{k:p_kt<z_\eps} (1-f(p_kt)) \geq \prod_{k=1}^\infty (1-f(z_\eps p_k))\\
  &=\exp\left\{\sum_{k=1}^\infty \log(1-f(p^k z_\eps))\right\} 
  \geq \exp\left\{-\sum_{k=1}^\infty f(p^k z_\eps)\right\}\\
  &=\exp\left\{-\frac{1}{2}\sum_{k=1}^\infty(1-e^{-2p^k
      z_\eps})\right\} 
  \geq \exp \left\{-\frac{1}{2}\sum_{k=1}^\infty 2p^k z_\eps\right\}\\
  &= \exp\left\{-\frac{z_\eps}{1-p}\right\}
\end{align*}
Therefore for any positive $t$,
$C_1<\Phi_1(t)<C_2$ with fixed and finite  positive constants~$C_1,C_2$.

As for the second factor $\Phi_2(t)$, if we denote $A(t)=\{k:p_kt\geq
z_\eps\}$, then for any $k\in A(t)$ we have $1-f(p_k t)\leq 1/(2-\eps)$, and
thus

\begin{displaymath}
  \Big(\frac{1}{2}\Big)^{|A(t)|}\leq \Phi_2(t) \leq
  \Big(\frac{1}{2-\eps}\Big)^{|A(t)|}.
\end{displaymath}
Since $|A(t)|=\card\{k:p_k\geq\frac{z_\eps}{t}\}=\card\{k:k<\frac{\log
  z_\eps}{\log p} -\frac{\log Ct}{\log p}\}=C_3+\lfloor \frac{\log
  t}{\log \frac{1}{p}}\rfloor$, we obtain
\begin{displaymath}
  C_4 \Big(\frac{1}{2}\Big)^\frac{\log t}{\log \frac{1}{p}}< \Phi_2(t) <
  C_5 \Big(\frac{1}{2-\eps}\Big)^\frac{\log t}{\log\frac{1}{p}},
\end{displaymath}
and finally,
\begin{displaymath}
  C_6 t^{-\frac{\log 2}{\log\frac{1}{p}}} <\prod\limits_{k=1}^\infty
  (1-f(p_k t)) <C_7 t^{-\frac{\log (2-\eps)}{\log\frac{1}{p}}},
\end{displaymath}
which yields the theorem statement for geometric $\{p_k\}$,
recalling an arbitrary small choice of $\eps$.

Moving to a general $\{p_k\}$, in case~(i) for all sufficiently large $k$,
$p_k < C_8 2^{-k} < 2^{C_9-k}$, and since $1-f(x)$ is non-increasing
in $x$, and $1-f(x)>1/2$ for $x>0$, we can choose a large enough $M$
and write
\begin{align*}
  \int_0^\infty \prod\limits_{k=1}^\infty(1-f(p_kt))\, dt&\geq C_{10}
  \int_0^\infty \prod_{k=M}^\infty (1-f(p_k t)) \, dt \\
  &\geq C_{10} \int_0 ^\infty (1-f(2^{C_9-k}t)) \, dt\\
  &= C_{10} \int_0^\infty \prod_{k=1}^\infty
  (1-f(2^{-k}\cdot2^{C_9+M-1}t)) \, \frac{d(2^{C_9+M-1}t)}{2^{C_9+M-1}}\\
  &= C_{11} \int_0^\infty \prod\limits_{k=1}^\infty (1-f(2^{-k}t)) \, dt.
\end{align*}
Similarly, in case~(ii), for all sufficiently large $k$, $p_k > C_{12}
(2-\eps)^{-k} > (2-\eps)^{C_{13}-k}$, 
and $1-f(x) \leq 1, x>0$, so we can choose a sufficiently large $M$ so that
\begin{displaymath}
  \int_0^\infty \prod\limits_{k=1}^\infty(1-f(p_kt)) \, dt \leq
  C_{14}\int_0^\infty \prod\limits_{k=1}^\infty (1-f((2-\eps)^{-k}t)) \, dt,
\end{displaymath}
and the theorem statement follows.
\end{proof}

We have seen that a BF model can be recurrent, but can it be
positive recurrent, i.e.\ can the number of steps to return to the
ground state have a finite expectation? The negative answer is
provided by Theorem~\ref{th:null-rec} which we prove next.

\begin{proof}[Proof of Theorem~\ref{th:null-rec}]\

 Introduce the following notation:
\begin{displaymath}
  \zeta^{\wedge m}_n=(\zeta^1_n,\dotsc,\zeta^m_n),\ \  \vec{0}^{\wedge m}=(
  \underbrace{0,\dotsc,0}_{m}), \ \tau_{\BF}^{\wedge m}=\inf\{n\in \N: \zeta^{\wedge
  m}_n=\vec{0}^{\wedge m}\}
\end{displaymath}
Obviously, $\tau_\BF=\inf\{n\in\N: \zeta_n=\vec{0}\}\geq
\tau_{\BF}^{\wedge m}$ almost surely.
Next, $\{\zeta^{\wedge m}_n\}_{n\geq 0}$ is an irreducible aperiodic (contrary to
$\{\zeta_n\}_{n\geq 0}$, which has period $2$) Markov chain with the
finite state space $\{0,1\}^m$ and a symmetric transition matrix, hence
its unique stationary distribution $\pi^{\wedge m}$, given by the solution of
the detailed balance equations
\begin{displaymath}
  \pi^{\wedge m}(x) p(x,y)= \pi^{\wedge m}(y) p(y,x),\ \  x,y\in\{0,1\}^m,
\end{displaymath}
is uniform on $\{0,1\}^m$.  Consequently, $\pi^{\wedge m}(\vec{0}^{\wedge
  m})=2^{-m}$ and $\E\tau_{\BF}^{\wedge m} = (\pi^{\wedge m}(\vec{0}^{\wedge
  m}))^{-1} = 2^m$.

Finally,
\begin{displaymath}
  \E\tau_\BF\geq\E\tau_{\BF}^{\wedge m} = 2^m\ \text{for every}\ m\in\N,
\end{displaymath}
finishing the proof.
\end{proof}

In order to prove Theorem~\ref{th:bf-moments}, we make use of Theorem~1 and
Corollary~1 in \cite{asp:96}. For convenience of the reader, we give
their formulation in our notation.

\begin{thm}[{\cite[Theorem 1]{asp:96}}]\label{th:asp-96}
  Suppose that $\{Y_n\}_{n\geq 0}$ is an $\{\fff_n\}$-adapted stochastic
  process taking values in an unbounded subset of $\R_+$. Introduce
  $\tau_A=\inf\{n\geq 0:Y_n\leq A\}$.  Suppose there exist
  positive constants $A$, $\eps$ such that for every $n$, $Y_n^{2r}$ is integrable and
  \begin{equation}\label{eq:cond-exp-est}
    Y^{2-2r}_n\E\left[Y_{n+1}^{2r}-Y_n^{2r}\cond \sF_n\right] \leq -\eps 
    \ \ \text{on\ }\{\tau_A\geq n\}.
  \end{equation}
  Then for any $r*$ satisfying $0<r^*<r$ there exists a constant $c=c(\eps,r^*,r)$
  such that for any $x\geq 0$
  \begin{displaymath}
    \E\tau^{r^*}_A\leq cx^{2r}\ \text{ whenever $Y_0=x$ a.s.}
  \end{displaymath}
\end{thm}

\begin{thm}[{\cite[Corollary 1]{asp:96}}]\label{th:asp-96-1}
  Let $\{Y_n\}_{n\geq 0}$, $\tau_A$ be as in Theorem~\ref{th:asp-96}. Suppose there exist
  positive constants $A,\eps,$ and $J$ such that for any n,
  \begin{displaymath}
    \E[Y^2_{n+1}-Y^2_n\cond\fff_n]\geq -\eps \ \text{on\ } \{\tau_A>n\}
  \end{displaymath}
  and, for some $\rho>1$,
  \begin{displaymath}
    Y_n^{2-2\rho}\E[Y^{2\rho}_{n+1}-Y^{2\rho}_n\cond\fff_n]\leq J\text{\ on\ } \{
    \tau_A>n\}.
  \end{displaymath}
  Suppose also that $Y_0=x>A$ and for some positive $r_0$ the process
  $\{Y^{2r_0}_{n\wedge\tau_A}\}_{n\geq 0}$ is a submartingale. Then for
  any $r>r_0$, $\E\tau^r_A=\infty$.
\end{thm}

We will also need the following technical Lemma. 

\begin{lem}\label{lem:coupling}
  Let $\{\zeta_n\}_{n\geq 0}$ be a discrete time BF model starting from
  the ground state $\zeta_0=\vec{0}$ with the parameter distribution
  $\ppp = \{p_1, p_2, \dotsc\}$ possibly with a finite support:
  $p_1\geq p_2\geq p_3 \geq \dotsc \geq 0$. 
  Then for $K = \min\{k:\ \sum_{i=k}^\infty p_i\leq 1/2\}$
  and any $n=1,2,\dotsc$, the 
    vector $(\zeta^K_n, \zeta^{K+1}_n,\dotsc)$ is
    stochastically dominated by the vector $(\czeta^K, \czeta^{K+1},\dotsc)$
    of i.i.d. $\Bern(1/2)$ random variables.
\end{lem}

\begin{proof}[Proof of Lemma~\ref{lem:coupling}]
  Assume that $\sum_{k=K}^\infty p_k > 0$, otherwise the Lemma statement is trivial.
  Let $\{\zeta_n\}_{n\geq 0}, \{\czeta_n\}_{n\geq 0}$ be two discrete
  time BF models with the same transition probabilities
  \begin{equation}\label{eq:evoprobs}
    \P\{\zeta^k_{n+1} = 1-\zeta^k_{n}\}= \P\{\czeta^k_{n+1} =
    1-\czeta^k_{n}\}=p_k,\  k\in\N,\ n=0,1,\dotsc,
  \end{equation} 
  where the first one starts from the ground state $\zeta_0=\vec{0}$
  and the second one starts from the stationarity: 
  $\czeta_0$ is a sequence of i.i.d.\ symmetric Bernoulli random variables:
  \begin{displaymath}
    \P\{\czeta^k_0 =1\}=1-\P\{\czeta^k_0 =0\}=1/2\quad \text{for all}\ k\in\N.
  \end{displaymath}
  Obviously, $0=\zeta^k_0\leq \czeta^k_0$ for all $k\in\N$ almost
  surely.  Our goal is to couple the Markov chains $\{\zeta_n\}$ and
  $\{\czeta_n\}$ on $\{0,1\}^{\N}$ preserving the almost sure
  coordinate-wise domination $\zeta^k_n\leq \czeta^k_n$ for
  all $n$ and all $k=K,K+1,\dotsc$.

  The idea is to treat the first $(K-1)$ bits of both Markov chains as
  a kind of `buffer' for which the domination does not generally
  holds. This is an expense to pay for the donination for the large coordinates.
  
  Specifically, we define the joint transition dynamics for $\zeta_n,
  \czeta_n$ inductively, for $n=0,1,2,\dotsc$.  Denote by $D_n$ the
  (random) set of discrepancies at time $n$, i.e.\ the set of indices
  $k\geq K$ at which $\zeta_n, \czeta_n$ disagree. The induction
  assumption is that the coordinate-wise domination is preserved on
  step $n$: $\zeta_n^k\leq \czeta_n^k$ for all $k\geq K$ and hence
  only discrepancies of the form $\zeta^k_n=0,\ \czeta^k_n=1$ are
  possible. Denote these by
  \begin{displaymath}
    D_n = \{k\geq K: \zeta^k_n=0,\ \czeta^k_n=1\}.
  \end{displaymath}
  The domination obviously holds for $n=0$.

  Let $F^{-1}(u)$ be the quantile function for the distribution $\mathcal{P}$:
  \begin{displaymath}
    F^{-1}(u) = \min\left\{k: \sum_{i=1}^k p_i > u\right\},\ \  u\in (0,1).
  \end{displaymath}
  The key element of the construction is a map $s_n(u):(0,1)\to(0,1)$
  which swaps the parts of $(0,1)$ corresponding to $D_n$ with the
  parts of $(0,1)$ of the same length, corresponding to the buffer:
  \begin{displaymath}
    s_n(u) = \begin{cases}
      1-u, & \text{if } F^{-1}(u)\in D_n, \text{ or } F^{-1}(1-u) \in D_n\\
      u, & \text{otherwise.}
    \end{cases}
  \end{displaymath}

  Introduce a common source of randomness for the chains: the sequence
  $U_1, U_2, \dotsc$ of i.i.d.~random variables distributed
  uniformly on the interval $(0,1)$. The indices of the bits to flip on step $(n+1)$
  in $\zeta_n$ and $\czeta_n$, $n=0,1,\dotsc$, are defined, respectively, as
  \begin{align*}
  \chi_{n+1} &= F^{-1}(U_{n+1}), \\
  \cchi_{n+1} &= F^{-1}(s_n(U_{n+1})).
  \end{align*}
  Since $s_n(u)$ preserves the Lebesgue measure, $s_n(U_{n+1})$ is
  also uniformly distributed implying that both chains have correct
  transition probabilities:
  \begin{displaymath}
  \P(\chi_{n+1} = k) = \P(\cchi_{n+1} = k) = p_k, k=1,2,\dotsc, n=0,1,2,\dotsc
\end{displaymath}
Moreover, if one of the chains is flipped at some coordinate $k\geq
K$, where the chains agree, the other one does the same. If,
otherwise, one of the chains is selected to be flipped at some
coordinate $k\geq K$, where the chains disagree, the other one is
flipped at one of the coordinates $k=1,2,\dotsc,K-1$ of the buffer. As
a result, the chains will agree at the flipped coordinate from
$D_n$. Thus no new discrepancies are created for $k\geq K$ and the
coordinate-wise domination $\zeta^k_{n+1}\leq \czeta^k_{n+1}$ is preserved
almost surely.

\end{proof}

\begin{proof}[Proof of Theorem~\ref{th:bf-moments}]

  Part (i).  Denote by $M_n$ the index of the rightmost active bit at
  time $n$: $M_n=\max\{k:\zeta^k_{n}=1\}$ with convention $M_n = 0$ for
  $\zeta_n = \mathbf{0}$. Put $Y_n^{2r}=y^{M_n}$ for some $y>1$ which
  will be selected later. Define the filtration $\fff_n=\sigma(
  \vec{\zeta}_0,M_1,\dotsc,M_n)$. The process $\{Y_n\}$ is obviously adapted to
  $\{\fff_n\}$. Recall that $\chi_k$ is an index of a bit flipped on step $k$, $\chi_k\sim
\mathcal{P}$ by the assumptions of (i). We have that
\begin{equation}\label{eq:expyto2r}
  \E (Y_n^{2r})= \E (y^{M_n})\leq \E (y^{\sum_{k=1}^n \chi_k})=\left(\E (y^{\chi_1})
  \right)^n.
\end{equation}
The inequality above follows, since $M_n\leq \max\{\chi_1, \chi_2,
\dotsc, \chi_n\} \leq \sum_{k=1}^n \chi_k$, so that the right-hand side
of~\eqref{eq:expyto2r} is finite whenever
\begin{equation}\label{eq:cond1}
  py<1.
\end{equation}
Next,
\begin{align*}
  \E[Y_{n+1}^{2r}-Y_n^{2r}\cond M_n=m]
  &=\underbrace{\E[(Y_{n+1}^{2r}-Y_n^{2r})\one_{\{\chi_{n+1}=m\}}\cond M_n=m]}_{E_1}\\
  &+\underbrace{\E[(Y_{n+1}^{2r}-Y_n^{2r})\one_{\{\chi_{n+1}>m\}}\cond M_n=m]}_{E_2}.
\end{align*}

Introduce $\psi(x_K\thru x_{m-1}) = y^{\max\{j:\ x_j=1,\ j=K\thru m-1\}}-y^m$.
Then
\begin{align}
  E_1 &\leq \E\left[\psi(\zeta_n^K, \dotsc, \zeta_n^{m-1})
    \one\{\chi_{n+1} = m\} \cond M_n = m\right]\nonumber\\
  &=p_m \E\left[\psi(\zeta_n^K, \dotsc, \zeta_n^{m-1}) \cond M_n = m\right] \label{eq:E1_1}
\end{align}
Our next step is to show that the vector $(\zeta_n^K, \dotsc, \zeta_n^{m-1})$ conditionally
on $\{M_n=m\}$ is stochastically dominated by a vector of i.i.d.\ Bernoulli random variables
$(\czeta^K, \dotsc, \czeta^{m-1})$.
Introduce an embedded Markov chain 
\begin{displaymath}
\Big\{\tzeta_l\Big\}_{l\geq 0} = \Big\{(\tzeta^1_l, \dotsc, \tzeta^{m-1}_{l})\Big\}_{l\geq 0}
\end{displaymath}
tracking the state of the first $(m-1)$ coordinates of $\{\zeta_n\}$
considered at the times when one of those coordinates changes.  We set
$\tzeta_0 = (\zeta^1_0, \dotsc, \zeta^{m-1}_0)$ and define
$\tzeta_l=(\zeta^1_{t_l(m)}, \dotsc, \zeta^{m-1}_{t_l(m)})$, where
$t_l(m)$ is the $l$th time when one of the first $(m-1)$ coordinates
of $\zeta_n$ is flipped. 

Lemma~\ref{lem:coupling} applied to the BF model $\{\tzeta_l\}_{l\geq
  0}$ with the corresponding flipping probabilities 
\begin{displaymath}
\widetilde{\ppp}
= \left\{\frac{p_1}{S_{m-1}}, \dotsc,
  \frac{p_{m-1}}{S_{m-1}}, 0, 0, \dotsc \right\},\quad S_{m-1}=\sum_{k=1}^{m-1}{p_k},
\end{displaymath}
implies for every $l=0,1,2,\dotsc$ the stochastic domination
\begin{displaymath}
  (\tzeta^{\tilde{K}}_l, \dotsc, \tzeta^{m-1}_l) \stleq
  (\czeta^{\tilde{K}}, \dotsc, \czeta^{m-1}),
\end{displaymath}
where $\czeta^{\tilde{K}},\dotsc,\czeta^{m-1}$ are i.i.d.\ $\Bern(1/2)$
random variables. Note that
\begin{displaymath}
   \tilde{K} = \min\{k:\ \sum_{i=k}^\infty p_i\leq 1/2\,
   S_{m-1}\}\leq K = \min\{k:\ \sum_{i=k}^\infty p_i\leq 1/2\}.
\end{displaymath}
Therefore, for every  $l=0,1,2,\dotsc$,
\begin{equation}\label{eq:stochdomtzeta}
(\tzeta^{K}_l, \dotsc, \tzeta^{m-1}_l) \stleq (\czeta^{K}, \dotsc,
\czeta^{m-1}).
\end{equation}

Introduce the series of events:
\begin{align*}
  A(n, m, l)  = &\{\text{by the time }n\text{ the first }m-1\text{
    coordinates of }\zeta \\
    &\text{ are flipped }l\text{ times}\}\\
   = &\Bigl\{\sum_{k=1}^n \one\{1\leq \chi_k\leq m-1\}=l\Bigr\},
\end{align*}
for $n=0,1\dotsc$, and $l=0,\dotsc, n$.  Conditionally on $A(n,m,l)$,
the distribution of $(\zeta^1_n, \dotsc, \zeta^{m-1}_n)$ is the same
as that of $(\tzeta^1_l, \dotsc, \tzeta^{m-1}_l)$, so we can
continue~\eqref{eq:E1_1} with:
\begin{align*}
  E_1 &\leq p_m
  \sum_{l=0}^{n}\E\left[\psi(\zeta_n^K,\dotsc,\zeta_n^{m-1})
    \one_{A(n, m, l)} \cond M_n=m\right]\\  
  &= p_m\sum_{l=0}^n \E\left[\psi(\zeta^K_n, \dotsc,
    \zeta^{m-1}_n)\one_{M_n=m} \cond A(n,m,l)\right]\,
  \frac{\P\left(A(n,m,l)\right)}{\P(M_n=m)}. 
\end{align*}
Now notice, that conditionally on $A(n,m,l)$, random variables\\
$\psi(\zeta^K_n, \dotsc, \zeta^{m-1}_n)=\psi(\tzeta^K_l, \dotsc,
\tzeta^{m-1}_l)$ and $\one_{M_n=m}$ are independent. Indeed, on
$A(n,m,l)$, the first variable is a function of the chain $\tzeta$
after $l$ steps which is governed by transition probabilities
$\widetilde{\ppp}$. While the event $M_n=m$ relates to configuration
of the bits $m,m+1,\dots$ after $n-l$ steps of the BF model with
parameter distribution $\{p_k/(1-S_{m-1}),\ k=m,m+1,\dots\}$. Therefore,
\begin{align*}
  E_1 \leq& p_m\sum_{l=0}^n \E\left[\psi(\zeta^K_n, \dotsc,
    \zeta^{m-1}_n) \cond A(n,m,l)\right] \times \\
    &\P\big(M_n=m \cond
    A(n,m,l)\big)\frac{\P\big(A(n,m,l)\big)}{\P(M_n=m)} \\ 
  =& p_m\sum_{l=0}^n \E\left[\psi(\tzeta^K_l, \dotsc,
    \tzeta^{m-1}_l)\right]\, \P\big(A(n,m,l)\cond M_n=m\big).
\end{align*}
The function $\psi$ is non-decreasing with respect to the
coordinate-wise order on its argument, so the stochastic
domination~\eqref{eq:stochdomtzeta} implies:
\begin{align*}
E_1 &\leq p_m
\E\psi(\czeta^K,\dotsc,\czeta^{m-1})\,\underbrace{\sum_{l=0}^n
  \P\left(A(n,m,l)\cond M_n=m\right)}_{=1} \\ 
 &=p_m \E\psi(\czeta^K,\dotsc,\czeta^{m-1}) \\
 &= \sum\limits_{k=K}^{m-1} (y^{k}-y^m)p_m\left(\frac{1}{2}\right)^{m-k}.
\end{align*}
Fix an arbitrary small $\eps>0$. Since $p_k \sim C_1 p^k, k\to\infty$,
one can, if necessary, increase $K$ so that $p_k\geq C_1(1-\eps) p^k$
for any $k\geq K$, and continue:
\begin{align*}
  E_1&\leq C_1(1-\eps)(py)^m
  \sum\limits_{k=K}^{m-1}((2y)^{-k}-2^{-m+k})\\
  &=C_1(1-\eps)(py)^m\left(\frac{2-2y}{2y-1} -\frac{(2y)^{-m+K}}{2y-1}
  +2^{-m+K}\right)\\
  &\leq C_1(1-\eps)(py)^m\left(\frac{2-2y}{2y-1}+2^{-m+K}\right).
\end{align*}
Because of our assumption $py<1$, and the asymptotic equivalence $p_k \sim C_1 p^k, k\to\infty$,
for an arbitrary small $\eps>0$ we can choose a large enough $M=M(\eps)$ so that
for any $m\geq M$:
\begin{align*}
  E_2&=\sum\limits_{k=1}^\infty p_{m+k}(y^{m+k}-y^m)\leq
  C_1(1+\eps)(py)^m\left(\sum\limits_{k=1}^\infty(py)^k-
  \sum\limits_{k=1}^\infty p^k\right)\\
  &=C_1(1+\eps)(py)^m\left(\frac{py}{1-py}-\frac{p}{1-p}\right).
\end{align*}
Introduce
\begin{displaymath}
  Q(p,y,\eps)=(1-\eps)\frac{2-2y}{2y-1}+(1+\eps)\Bigl(\frac{py}{1-py}-\frac{p}{1-p}\Bigr).
\end{displaymath}
Then $Y_n^{2r}=y^{M_n}$ yields
\begin{displaymath}
  Y_n^{2-2r}\E\left[Y_{n+1}^{2r}-Y_n^{2r}\cond M_n=m\right]
  \leq C_1(Q(p,y,\eps)+(1-\eps)2^{-m+2})(py^{\frac{1}{r}})^m.
\end{displaymath}
Now fix a $p<1/2$. For the last expression to be negative and separated
from zero for all $m$ large enough, it is necessary for $Q(p,y,\eps)$ to be
negative and for $py^{\frac{1}{r}}$ to be greater than one.
However, $Q(p,y,\eps)<0$ reduces to
\begin{displaymath}
  \begin{cases}
    0<p<1/2,\\
    1<y<\frac{1}{2p},\\
    0<\eps<\frac{2p^2y-4py-p+2}{2p^2y-3p+2}.
  \end{cases}
\end{displaymath}
The right part of the third inequality is positive whenever the first two
inequalities are satisfied. 
Putting it all together, for a fixed pair of $p,r$ we can pick $y$ and $M$ so
that $Y_n=y^{\frac{M_n}{2r}}$, given that $M_0>M$,
satisfies the conditions of Theorem~\ref{th:asp-96}
if and only if the following system of inequalities can be solved for $y$:
\begin{displaymath}
\begin{cases}
  1<y<\frac{1}{2p},\\
  py^{\frac{1}{r}}>1,
\end{cases}
\end{displaymath}
and the latter is possible when $r<1-\frac{\log 2}{\log\frac{1}{p}}$.

Denote $\tau_{x}=\inf\{n\geq 1:M_n\leq x\}$. Then 
Theorem~\ref{th:asp-96} implies that for $p<1/2$ and $r<1-\frac{\log 2}
{\log\frac{1}{p}}$ there exists $C=C(p,r)$ such that for a
particular choice of $y, M$ we have
\begin{displaymath}
  \E[\tau_M^r\cond M_0=x]\leq Cy^x\leq C\left(\frac{1}{2p}\right)^x.
\end{displaymath}
Now we prove that $\tau_\BF^r=\tau_0^r$ is integrable and satisfies the same asymptotic bound.
In $\E[\tau_0^r|M_0=x]$, $\tau_0$ is the first time when the process $M_n$ reaches
$0$ starting from the state $x$. For any $M\geq 0$ we have $\tau_0 = \tau_{M} + (\tau_M - \tau_0)$.
By simple coupling arguments,
the law of $(\tau_M - \tau_0)$, conditional on $\{M_0=x\}$, is stochastically dominated by 
the law of $\tau_0$, conditional on $\{M_0=M\}$. That, together with
the inequality $(a+b)^r\leq 2^r(a^r+b^r)$ for $0<r<1$ and non-negative $a,b$, gives
the bound:
\begin{displaymath}
\E[\tau_0^r\cond M_0=x]\leq 2^r\big(\E[\tau_M^r\cond
M_0=x]+\E[\tau_{0}^r\cond M_0=M]\big).
\end{displaymath}
We have just obtained an asymptotic upper bound for the first conditional expectation
under the parentheses. It is left now to show that the second
expectation is also bounded.  Conditionally on $\{M_0=M\}$,
$\tau_0$ is stochastically dominated from above by the sum of two
terms. The first one is the time needed
for $\vec{\zeta}_n$ to reach $\vec{0}$ not leaving
the finite sub-cube $\{0,1\}^M$, which is in turn dominated by $\tau^{\wedge M}_0
=\inf\{n:\zeta_n^{\wedge M} = \vec{0}^{\wedge M}\}$.
The second one is a geometrically 
distributed number of excursions $\gamma\sim\mathrm{Geom}(\pi)$ from
$\{0,1\}^M$. Thus
\begin{displaymath}
  \E[\tau_0^r\cond M_0=M]\leq\sum\limits_{k=1}^{\infty}
  \E[\tau_0^r\cond M_0=M,\gamma=k]\,\P\{\gamma=k\}.
\end{displaymath}
Now, conditionally on $\{\gamma=k\}$,             
\begin{align*}
  \E[\tau^r_0\cond M_0=M,\gamma=k]&\leq \E\left[\Big(\tau_0^{\wedge M}+
  \sum\limits_{j=1}^k\psi_j\Big)^r \cond M_0=M\right]\\
  &\leq k^{1+r}(\E[(\tau_0^{\wedge M})^r\cond M_0=M]+
  \E\psi^r),
\end{align*}
where $\psi_j$  is the length of excursion $j=1,\dotsc,\gamma$ and
$\psi$ stands for length of a typical excursion.
The first expectation inside the parentheses is a finite constant. As for
the second, for some constant $C_2>0$ we have
\begin{align*}
  \E\psi^r&=1+\sum\limits_{k=1}^\infty p_{k+M}\E[\tau_M^r\cond M_0=k+M]\\
  &\leq 1+\sum\limits_{k=1}^\infty C_2p^{k+M}\left(\frac{1}{2p}\right)^{k+M}
  <\infty.
\end{align*}
Thus for some constant $C_3>0$
\begin{displaymath}
  \E[\tau_0^r\cond M_0=M]\leq \sum\limits_{k=1}^\infty C_3 k^{1+r} \pi (1-\pi)^{k-1}<
  \infty,
\end{displaymath}
finishing the proof of part (i).

Proof of part~(ii). Put $Y^2_n=y^{M_n}$ for some $y>1$ and check the conditions of
Theorem~\ref{th:asp-96-1}.  As before, $Y_n$ is adapted and
for an arbitrary small $\eps>0$ we can choose $M=M(\eps)$ large enough
so that:
\begin{align*}
  \E[Y^2_{n+1}-Y^2_n\cond M_n=m]\geq& -p_m y^m +\sum\limits_{k=1}^\infty
  p_{m+k}(y^{m+k}-y^m)\\
  =&-C_1(1-\eps)p^m y^m\\
  &+\sum\limits_{k=1}^\infty
  C_1(1+\eps)p^{m+k}(y^{m+k}-y^m)\\
  =&C_1(py)^m(-1+\eps+(1+\eps)\sum\limits_{k=1}^\infty p^k(y^k-1))\\
  =&C_1(py)^m\left(-1+\eps+\frac{(1+\eps)p(-1+y)}{(1-p)(1-py)}\right)\\
  =&C_1(py)^m R(p, y, \eps),
\end{align*}
where $R(p, y, \eps) =
\left(-1+\eps+\frac{(1+\eps)p(-1+y)}{(1-p)(1-py)}\right)$.  It is then
possible to choose a small enough $\eps>0$ and a large $M$ so that the
latter expression is bounded from below for all $m>M$, when $py<1$.
Furthermore, for such $p,y$ we have as before:
\begin{align*}
  Y_n^{2-2\rho}\E[Y_{n+1}^{2\rho}-Y_n^{2\rho}\cond M_n=m]\leq
  &C_1y^{m(1-\rho)}(py^\rho)^m\times \\ &(Q(p,y^\rho,\eps)+(1-\eps)2^{-m+K})
\end{align*}
which is bounded from above when $\rho$ is such that $py^\rho<1$ (such
a $\rho>1$ exists whenever $py<1$).

Finally, check for which $r_0$ the process $Y_{n\wedge \tau_M}^{2r_0}$
is a submartingale. Since
\begin{displaymath}
  \E[Y^{2r_0}_{n+1}-Y^{2r_0}_n\cond M_n=m]\geq C_1(py^{r_0})^m
  R(p, y^{r_0}, \eps),
\end{displaymath}
we can choose $\eps>0$ so that the latter is greater than zero for any $m>M$,
if $r_0\in(\frac{\log\frac{1}{2p-p^2}}{\log y},1)$. Recalling that we can take
$y$ arbitrary close to $1/p$, we conclude that the conditions of
Theorem~\ref{th:asp-96-1} are satisfied for any $r_0$ such that
$r_0\in(1-\frac{\log(2-p)}{\log\frac{1}{p}},1)$. Adding this
together with the results of Theorem~\ref{th:null-rec} implies that
none of the fractional
moments of $\tau_M$ (and hence of $\tau_0$) of order higher than
$1-\frac{\log(2-p)}{\log\frac{1}{p}}$ exists, finishing the proof of
Part~(ii).
\end{proof}

\subsection{Transience and recurrence of DB model}

\begin{proof}[Proof of Theorem~\ref{th:db-gen-exp}]

  (i) Consider the discrete-time version of the DB model. Introduce   $R_n$ --
  the index of the rightmost bit (i.e.\ with the largest   index) that has ever
  been flipped by time $n$. The sequence   $\{R_n\}$ is a.s.\ non-decreasing.
  We aim to prove that almost   surely for infinitely many terms of the sequence
  $\{R_n\}$, each of   the bits $1, 2, \dotsc, R_n$ is flipped at least twice
  before the   next flip of some bit with an index larger than $R_n$. That would
  guarantee that the ground state of the DB model, corresponding to   the set of
  states   $$\left\{y\in \{0,1,2\}^\N: y \text{ has no 1's and only a finite
    number of 2's}\right\}$$ of Markov chain $\{\zeta_n\}$, is visited infinitely
  often.

  It is convenient to use the continuous-time representation now.  Let
  $\Pi_1(t), \Pi_2(t),\dotsc$ be the sequence of independent Poisson
  processes (clocks) describing the times at which, respectively, the
  $1$st, the $2$nd, etc.\ bits are flipped. Introduce
  $\tau_{>k}=\inf\{t>0:\sum_{j=k+1}^\infty\Pi_j(t)>0\}$, the time of
  the first flip of a bit with an index greater than $k$. Note that
  $\tau_{>k}$ is a stopping time for each $k=0,1,2,\dotsc$, and,
  moreover,
\begin{displaymath}
  \tau_{>1}\leq \tau_{>2}\leq \tau_{>3} \leq \dotsc
\end{displaymath}

Introduce the events
\begin{align*}
  &A_k & = & &\{k\text{th } &\text{bit appears in the sequence }\{R_n\}\}, \nonumber
\\  &B_k & = & &A_k\cap\{ & \text{starting from the first flip of }k\text{'th bit,
    each of the bits } \nonumber
\\   & & & & &1, 2, \dotsc, k\text{ is }
  \text{flipped at least twice before the first flip}
\\ & & & & & \text{of one of the
    bits }k+1, k+2, \dotsc\} \nonumber
\end{align*}
Our aim is to prove that the events $B_k$ happen infinitely often.
In terms of a continuous-time notation, we can rewrite:
\begin{align}
  A_k&=\{\tau_{>k-1}<\tau_{>k}\},\nonumber\\
  \label{eq:B_k-def}B_k&=\bigcap_{j\leq k}\{\Pi_j(\tau_{>k-1},\tau_{>k})\geq 2\},
\end{align}
where $\Pi(t_1,t_2)$ stands for the number of points a Poisson process
$\Pi$ has in $(t_1,t_2)$.  Since $\{\tau_{>k}\}$ is a sequence of
stopping times, it is not hard now to see that the events $B_k$ are
independent of each other. By the Borel--Cantelli Lemma it suffices to
prove that the series $\sum_{k\geq 1}\P\{B_k\}$ diverges.

The probability of $A_k$ (probability of an index $k$ to ever appear
in the sequence $\{R_n\}$) is exactly $p_k/(p_k+Q_k) =
1-Q_{k}/Q_{k-1}$, which is uniformly bounded away from zero given
assumptions of (i).

As follows from~$\eqref{eq:B_k-def}$, the probability $\P(B_k\cond
A_k)$ is equal to the probability for each of the first $k$ Poisson
clocks $\Pi_1(t), \dotsc, \Pi_k(t)$ to tick at least twice before the
time of the first tick of one of the clocks $\Pi_{k+1}(t),
\Pi_{k+2}(t), \dotsc$ We write:
\begin{align}
  \P(B_k\cond A_k) &= \P\Bigl(\bigcap_{j=1}^{k}\{\Pi_j(\taugre)\geq 2\}\Bigr)\nonumber\\
  &=\int_0^\infty \prod_{j=1}^k \P\{\Pi_j(t)\geq 2\}\, d\P(\taugre\leq
  t)\label{eq:bkcondak} 
\end{align}
Introduce $g(x)=e^{-x}(1+x)$. Now, due to (i), there exists a large
$K$ such that $\text{ for any } k\geq j\geq K$:
\begin{displaymath}
\frac{p_j}{Q_k} =
\frac{p_j}{Q_{j-1}}\frac{Q_{j-1}}{Q_j}\dotsc\frac{Q_{k-1}}{Q_k} 
\geq \left(\frac{1}{p}-1\right)\underbrace{\cdot \frac{1}{p}\dotsc
  \frac{1}{p}}_{k-j+1} = C_2 p^{j-k}. 
\end{displaymath}
The function $g(x)$ is strictly decreasing in
$x$, so we can continue~\eqref{eq:bkcondak} :
\begin{align*}
  & =\int_0^\infty \prod_{j=1}^k(1-g(p_jt)) Q_k e^{-Q_kt}\,dt \\
  & = \int_0^\infty \prod_{j=1}^k(1-g(\frac{p_j}{Q_k}t))e^{-t}\,dt\\
  & \geq C_1\int_0^\infty \prod_{j=1}^{k-K}(1-g(C_2 p^{-j}t))
  e^{-t}\,dt, \text{ for } k\geq K, 
\end{align*}
where $C_1, C_2$ are positive constants.
Next,
\begin{displaymath}
  \prod_{j=1}^{k-K}(1-g(C_2 p^{-j}t))\geq \prod_{j=1}^\infty(1-g(C_2
  p^{-j}t)). 
\end{displaymath}
Show that the latter is strictly positive:
\begin{displaymath}
  \sum_{j=1}^\infty g(C_2tp^{-j}) = \sum_{j=1}^\infty e^{-C_2tp^{-j}}(1+C_2tp^{-j})
  \leq C_3\sum_{j=1}^\infty e^{-C_4 tp^{-j}}<\infty
\end{displaymath}
for all $t$, thus $\prod_{j=1}^{k-K}(1-g(C_2 p^{-k}t))$ is bounded away from
zero uniformly in $k, k\geq K,$ by $h(t)=\prod_{j=1}^\infty(1-g(C_2 p^{-j}t))>0$,
and $\P(B_k\cond A_k)\geq C_1\int_0^\infty h(t) e^{-t}\,dt > 0$, so the series
$\sum_{k=1}^\infty \P(B_k)$ diverges and the DB model is recurrent under 
the assumptions of~(i).

(ii) Now, assume that $p_k\sim C e^{-\alpha k^\gamma}$. Consider the total time
$\nu$ spent in the ground state, when none of the bits is active. We
are going to prove for this particular choice of $p_k$ that the
expected time spent in the ground state
\begin{displaymath}
 \E\nu=\int_0^\infty \prod\limits_{k=1}^\infty(1-p_kte^{-p_kt})\, dt
\end{displaymath}
is finite. The product under the integral is bounded by
 \begin{displaymath}
\prod\limits_{k=1}^\infty(1-p_kte^{-p_kt}) \leq \exp\Big \{\card\{k:
l_{1,\eps}\leq p_kt \leq l_{2,\eps}\} \log (1-1/e+\eps)\Big\}.
\end{displaymath}
Here $l_{1,\eps},l_{2,\eps}$ are the left and the right boundaries of the interval,
where the function $xe^{-x}$ is greater or equal than $1/e-\eps$.
Taking into account the particular choice of $p_k$, we write:
\begin{align}
  \card\{k:l_{1,\eps}\leq p_kt \leq l_{2,\eps}\} &\sim 
  \left(\frac{1}{\alpha}\log\frac{tC}{l_{1,\eps}}\right)^{1/\gamma}
  -\left(\frac{1}{\alpha}\log\frac{tC}{l_{2,\eps}}\right)^{1/\gamma} \nonumber\\
  &\sim\frac{\log l_{2,\eps}-\log
    l_{1,\eps}}{\gamma\alpha^{\frac{1}{\gamma}-1}}(\log
  (tC))^{\frac{1}{\gamma}-1}, 
  \label{eq:cardbound}
\end{align}
hence the infinite product in question is integrable for $\gamma<1/2$.
\end{proof}

\begin{rem}
  The sufficient condition in Theorem~\ref{th:db-gen-exp}, (i) is
  slightly stronger than a condition similar to the one in
  Theorem~\ref{th:bf-gen-suf}, (i):
  \begin{equation}\label{eq:counterexample}
    \limsup_{k\to\infty} \beta^k p_k <\infty \text{ for some constant }
    \beta > 1. 
  \end{equation}
  It is not hard to see that the assumption of
  Theorem~\ref{th:db-gen-exp}, (i) implies~\eqref{eq:counterexample}
  for $\beta=1/(p+\eps)$ for any $\eps\in (0,1-p)$.  The converse
  implication is not true in general, for a counterexample we can put
  \begin{displaymath}
    \kappa(k) = \min\{j^2 : j\in \N \text{ and } j^2 > k\},
  \end{displaymath}
  and then
  \begin{displaymath}
    p_k = C 2^{-\kappa(k)}, \ \ k=1,2,\dotsc
  \end{displaymath}
  for a suitable constant $C$. Then~\eqref{eq:counterexample} holds
  with $\beta=2$.  The assumption of Theorem~\ref{th:db-gen-exp}, (i)
  fails to hold: for the subsequence $k_i = i^2, i=1,2,\dotsc$ we have
  \begin{align*}
    \frac{Q_{k_i}}{Q_{k_i-1}}&=1-\frac{p_{k_i}}{Q_{k_{i}-1}} =1-
    \frac{p_{i^2}}{\sum_{j=k_i}^{\infty} p_j} \geq
    1-\frac{p_{i^2}}{\sum_{j=i^2}^{(i+1)^2-1} p_{j}}\\
    &= 1- \frac{C 2^{-(i+1)}}{((i+1)^2-i^2)C 2^{-(i+1)}} \geq 1-
    \frac{1}{2i+1} \to 1, \ \ i\to\infty.
  \end{align*}
  However, the converse implication will hold if in addition
  to~\eqref{eq:counterexample} we require, for instance, the sequence
  $\{Q_k/Q_{k+1}\}$ to be monotone.
\end{rem}

\subsection{The Central Limit Theorem}

For the proof of the CLT for the number of active bits in BF and DB models
we use the following general CLT for
the triangular array, see, e.g.,\ {\cite[Ch.8,~Theorem 5]{Bor:98}}:
\begin{thm}
\label{th:bor-tri}
  Let $\{\xi_{k,n}\},\ 1\leq k\leq r_n,\ 1\leq n\leq \infty$
  be a triangular array of random variables such that $\E\xi_{k,n}=0$
  and that the random variables $(\xi_{k,n})_{1\leq k\leq r_n}$ are mutually independent
  inside of every row $n=1,2,\dotsc$. Assume that:
  \begin{itemize}
  \item[(i)]\label{th-b-i}$\sum\limits_{k=1}^{r_n}\E \xi_{k,n}^2=1,$\\
  \item[(ii)]\label{th-b-ii}$\sum\limits_{k=1}^{r_n}\E\big[\xi_{k,n}^2;\,
    |\xi_{k,n}|>
    M\big]\to 0,\ n\to\infty, \text{\ for every\ } M>0.$
  \end{itemize}
  Then
  \begin{displaymath}
    \sum_{k=1}^{r_n}\xi_{k,n} \stackrel{D}{\Longrightarrow} \mathcal{N}(0,1), \text{as\ }
      n\to\infty
  \end{displaymath}
\end{thm}

\begin{proof}[Proof of Theorem~\ref{th:bs-db-clt}]
It is easy to see that the expected number of active bits $\E\n_t$ in BF model
tends to infinity. We can write $\E\n_t$ explicitly as
\begin{displaymath}
\E\n_t = \sum_{k=1}^\infty \P\{\zeta_t^k = 1\} =
\sum_{k=1}^\infty \frac{1}{2}(1-e^{-2p_kt}).
\end{displaymath}
Every term in the latter sum monotonously approaches $1/2$ as $t\to \infty$, thus
the whole sum tends to infinity.

Next, for the DB model, given the assumption~\eqref{eq:sufdbinfexp}, if we fix a small
$\eps>0$ and take $l_{1,\eps}, l_{2,\eps}$ to be as in~\eqref{eq:cardbound}
the left and the right borders
of the interval where the function $xe^{-x}$ is greater than $1/e-\eps$,
then, by the same reasoning as in~\eqref{eq:cardbound}, we obtain:
\begin{align*}
\E\n_t &= \sum_{k=1}^\infty \P\{\zeta_t^k = 1\} =
\sum_{k=1}^\infty p_k te^{-p_kt}\\ 
&\geq (e^{-1}-\eps)\card\{k: \lambda_{1,\eps} \leq p_kt \leq
\lambda_{2,\eps}\} 
\geq C_1(\log (tC))^{\frac{1}{\gamma}-1} \to \infty
\end{align*}
for a constant $C_1$ depending on $\eps$, $\gamma$ and $\alpha$.

The rest of the proof works for both BF and DB models.  It is
sufficient to prove the CLT for the embedded discrete time process
$\{\n_{T_n}\}_{n\geq 1}$ for an arbitrary non-random time sequence
$\{T_n\}_{n\geq 1}$ going to infinity.  Let us fix such a sequence and
denote $\zeta_n := \zeta_{T_n}$ and $\n_n:=\n_{T_n}$, for short.
Introduce random variables
\begin{align*}
  Z_{n,k} & =\one\{\zeta_{n}^k = 1\},\\
  \xi_{n,k} & =\begin{cases}
    \frac{Z_{n,k}-\E Z_{n,k}}{\sqrt{\var \n_n}}, & k<r_n, \\\rule{0mm}{8mm}
    \frac{\sum\limits_{k\geq r_n}(Z_{n,k}-\E Z_{n,k})}{\sqrt{
      \var\n_n}}, & k=r_n.
  \end{cases}
\end{align*}
We leave ourselves a freedom to choose a suitable sequence $\{r_n\}$
later.  Check the conditions of Theorem~\ref{th:bor-tri}. The random
variables $\{\xi_{n,k}\}_{k=1}^{r_n}$ are mutually independent for
every $n$. Condition~(i) holds trivially. As for (ii), one has:
\begin{align}
  \sum_{1\leq k\leq r_n} & \E\big[\xi_{n,k}^2;\,|\xi_{n,k}|>M\big]=
  \overbrace{\sum_{1\leq k\leq r_n-1}\E\big[\xi_{n,k}^2;\,|\xi_{n,k}|>
  M\big]}^{S_1} \notag \\
  & +\underbrace{\E\big[\xi_{n,r_n}^2;\,|\xi_{r_n,n}|>M\big]}_{S_2}. \label{eq:clt-sum-decompose}
\end{align}
By the assumptions $\E\n(t)\to\infty$ as $t\to\infty$.
Moreover,
\begin{displaymath}
  C_2\E\n(t)\leq\var \n(t)=\sum_{k\geq 1}f(p_kt)(1-f(p_kt))\leq\E\n(t),
\end{displaymath}
where $f(x) = \frac{1}{2}(1-e^{-x})$ in BF model, $f(x) = xe^{-x}$ in
DB model, and $C_2=(1-\sup_{x\in \R^+} f(x))$, with the respective
$f$, so that $0<C_2<1$ in both cases.  Hence, by construction of
$\xi_{n,k}$, the sum $S_1$ in~\eqref{eq:clt-sum-decompose} tends to
$0$ as $n$ goes to infinity, because almost surely $\xi_{n,k}\leq
1/\var \n_n\to 0$ and every term in $S_1$ is eventually zero.
Lastly,
\begin{displaymath}
  \E\xi_{r_n,n}^2=\frac{1}{\var\n_n}\sum_{k\geq r_n}f(p_kT_n)
  (1-f(p_kT_n))
\end{displaymath}
and so we can choose such $r_n$ that the
latter sum is no larger than, for instance, $\sqrt{\var \n_n}$, thus
satisfying Condition~(ii) of Theorem~\ref{th:bor-tri} and finishing the proof.
\end{proof}

\begin{Ack} The authors thank Sergey Foss for the discussions from which the
Bit Flipping models we consider here started, as well as for the
follow-up talks and insights on relation of Bit Flipping to other
fields.  The authors are grateful to Robin Pemantle for an idea of a
continuous-time implementation of the process, which proved to be an
irreplaceable tool in the analysis. We are grateful to two anonymous
referees for thorough reading and their valuable comments which
allowed us to significantly improve the presentation of the material.
\end{Ack}

\newpage

\bibliographystyle{plain}
\bibliography{bitflip}

\end{document}